\documentclass[11pt]{article}

\usepackage{epic,eepic,epsfig,amssymb,amsmath,amsthm,graphics}

\usepackage{multimedia}

\usepackage{pdftricks}

\begin{psinputs}
\usepackage{pstricks,pst-text}
\end{psinputs}

\usepackage{amsmath,amssymb,amsthm,amsfonts,mathrsfs}

\def\R{\mathbb R}
\def\G{\mathbb G}

\def\H{\mathbb H}

\newcommand{\bigzero}{\mbox{\normalfont\Large\bfseries 0}}
\def\MCP{{\rm MCP}}

\numberwithin{equation}{section}

\newtheorem{theorem}{Theorem}

\newtheorem{lemma}[theorem]{Lemma}
\newtheorem{proposition}[theorem]{Proposition}

\newtheorem{definition}[theorem]{Definition\rm}
\newtheorem{remark}[theorem]{Remark}

\makeindex

\begin{document}

\title{Measure contraction properties for two-step analytic sub-Riemannian structures and Lipschitz Carnot groups}
\footnotetext{Both authors are supported by the ANR project SRGI ``Sub-Riemannian Geometry and Interactions'', ANR-15-CE40-0018. }


\author{Z.~Badreddine\thanks{Universit\'e C\^ote d'Azur, CNRS, Inria, Labo. J.-A. Dieudonn\'e, UMR CNRS 7351, Parc Valrose 06108 Nice, Cedex 2, France ({\tt Zeinab.Badreddine@unice.fr})} \and L.~Rifford\thanks{Universit\'e C\^ote d'Azur, CNRS, Inria, Labo. J.-A. Dieudonn\'e, UMR CNRS 7351, Parc Valrose 06108 Nice, Cedex 2, France ({\tt
      Ludovic.Rifford@math.cnrs.fr})}}



\maketitle

\begin{abstract}
We prove that two-step analytic sub-Riemannian structures on a compact analytic manifold equipped with a smooth measure and Lipschitz Carnot groups satisfy measure contraction properties.
\end{abstract}

\section{Introduction}\label{SECintroduction}
\label{Intro}
The aim of this paper is to provide new examples of sub-Riemannian structures satisfying measure contraction properties. Let $M$ be a smooth manifold of dimension $n\geq 3$ equipped with a sub-Riemannian structure $(\Delta,g)$ of rank $m<n$, whose geodesic distance $d_{SR}$ is supposed to be complete. We refer the reader to Appendix \ref{Notations} for the notations used throughout the paper. As in the previous paper of the second author on the same subject \cite{rifford13}, we restrict our attention to the notion of measure contraction properties in metric measured spaces with negligeable cut loci (if $A\subset M$ is a Borel set then $\mathcal{L}^n (A) =0$ means that $A$ has vanishing $n$-dimensional Lebesgue measures in charts):

\begin{definition}
We say that the sub-Riemannian structure $(\Delta,g)$ on $M$ has negligeable cut loci if for every $x\in M$, there is a measurable set $\mathcal{C}(x) \subset M$ with  
$$
\mathcal{L}^n \left( \mathcal{C}(x)\right) =0,
$$
and a measurable map $\gamma_{x} \, : \, \left( M \setminus \mathcal{C}(x)\right) \times [0,1] \longrightarrow M$ such that for every $y \in  M \setminus \mathcal{C}(x)$ the curve 
$$
s \in [0,1] \longmapsto \gamma_x (s,y)
$$
is the unique minimizing geodesic from $x$ to $y$. 
\end{definition}

Measure contraction properties consists in comparing the contraction of volumes along minimizing geodesics from a given point with what happens in classical model spaces of Riemannian geometry. We recall that  for every $K\in \R$, the comparison function $s_K: [0,+\infty) \rightarrow [0,+\infty)$ ($s_K: [0,\pi/\sqrt{K})\rightarrow [0,+\infty)$ if $K>0$) is defined by
$$
s_K(t) := \left\{ \begin{array}{lll}
\frac{\sin(\sqrt{K}t)}{\sqrt{K}} & \mbox{ if } K>0\\
t & \mbox{ if } K=0 \\
\frac{\sinh(\sqrt{-K}t)}{\sqrt{-K}} & \mbox{ if } K<0.
\end{array}
\right.
$$
In our setting, the following definition is equivalent to the notion of measure contraction property introduced by Ohta in \cite{ohta07}  for more general measured metric spaces (see also \cite{sturm06b}). 

\begin{definition}
Let $(\Delta,g)$ be a sub-Riemannian structure on $M$ with negligeable cut loci, $\mu$ a measure absolutely continuous with respect to $\mathcal{L}^n$ and $K\in \R, N> 1$ be fixed. We say that $(\Delta,g)$ equipped with $\mu$ satisfies $\MCP(K,N)$ if for every $x\in M$ and every measurable set $A\subset M \setminus \mathcal{C}(x)$ (provided that $A\subset B_{SR}(x,\pi\sqrt{N-1/K})$ if $K>0$) with $0< \mu(A)<\infty$,
$$
\mu \left(A_s\right)  \geq  \int_{A} s \left[  \frac{s_K\left(sd_{SR}(x,z)/\sqrt{N-1}\right)}{s_K\left(d_{SR}(x,z)/\sqrt{N-1}\right)} \right]^{N-1}  \, d\mu(z) \qquad \forall s \in [0,1],
$$
where $A_s$ is the $s$-interpolation of $A$ from $x$ defined by
$$
A_s := \Bigl\{ \gamma_x(s,y) \, \vert \, y \in A \setminus \mathcal{C}(x) \Bigr\} \qquad \forall s \in [0,1].
$$
In particular, $(\Delta,g)$ equipped with $\mu$  satisfies $\MCP(0,N)$ if for every $x\in M$ and every measurable set $A\subset M\setminus \mathcal{C}(x)$ with $0< \mu(A)<\infty$,
$$
\mu \left(A_s\right) \geq s^N \mu (A) \qquad \forall s \in [0,1].
$$
\end{definition}

To our knowledge, the first study of measure contraction properties in the sub-Riemannian setting has been performed by Juillet in his thesis. In \cite{juillet09}, Juillet proved that the $n$-th Heisenberg group $\H^n$ (with $n\geq 1$) equipped with its sub-Riemannian distance and the Lebesgue measure $\mathcal{L}^{2n+1}$ (in this case the ambiant space is $\R^{2n+1}$) satisfies $\mbox{MCP}(0,2n+3)$. This result is sharp for two reasons. First, Juillet proved that $\H^n$ does not satisfy any other stronger notion of "Ricci curvature bounded from below" in metric measured spaces such as for example the so-called curvature dimension property (see \cite{lv07,sturm06a,sturm06b,villanibook}). Secondly, Juillet showed that $2n+3$ is the optimal dimension for which $\H^n$ satisfies $\mbox{MCP}(0,N)$, there is no $N<2n+3$ such that $\H^n$ (equipped with $d_{SR}$ and $\mathcal{L}^{2n+1}$) satisfies $\mbox{MCP}(0,N)$. The Juillet's Theorem, which settled the case of the simplest sub-Riemannian structures, paved the way to the study of measure contraction properties for more general sub-Riemannian structures. In \cite{al14}, Agrachev and Lee investigated the case of sub-Riemannian structures associated with contact distributions in dimension $3$. In \cite{lee16,lcz16}, Lee and Lee, Li and Zelenko studied the particular case of Sasakian manifolds. In \cite{rifford13}, the second author proved that any ideal Carnot group satisfy $\mbox{MCP}(0,N)$ for some $N>1$ (it has been shown later by Rizzi \cite{rizzi16} that a Carnot group is ideal if and only if it is fat). In \cite{rizzi16}, Rizzi showed that any co-rank $1$ Carnot group of dimension $k+1$ (equipped with the sub-Riemannian distance and a left-invariant measure) satisfies $\mbox{MCP}(0,k+3)$. Finally, more recently, Barilari and Rizzi \cite{br17} proved that $H$-type Carnot groups of rank $k$ and dimension $n$ satisfy $\mbox{MCP}(0,k+3(n-k))$. The purpose of the present paper is to pursue the qualitative approach initiated by the second author in \cite{rifford13}. We aim to show that some assumptions on the sub-Riemannian structure insure that the sub-Riemannian distance enjoyes some properties which guarantee that some measure contraction property of the form $\mbox{MCP}(0,N)$ is satisfied for some $N>1$ (in fact $N$ has to be greater or equal to the geodesic dimension of the sub-Riemannian structure as introduced by Rizzi \cite{rizzithesis}). Our approach is purely qualitative, we do not compute any curvature type quantity in order to find the best exponents. Our results are concerned with two-step analytic sub-Riemannian structures and Lipschitz Carnot groups.\\

Given a (real) analytic manifold $M$, we say that $(\Delta,g)$ is analytic if both $\Delta$ and $g$ are analytic on $M$. Moreover, we recall that a distribution $\Delta$, or a sub-Riemannian structure $(\Delta,g)$,  is {\it two-step} if 
$$
[\Delta,\Delta](x):= \Bigl\{ [X,Y](x) \, \vert \, X, Y \mbox{ smooth sections of } \Delta \Bigr\} =T_xM \qquad \forall x \in M.
$$
A measure on $M$ is called {\it smooth} if it is locally defined by a positive smooth density times the Lebesgue measure $\mathcal{L}^n$, our first result is the following:

\begin{theorem}\label{THMgen}
Every two-step analytic sub-Riemannian structure on a compact analytic manifold equipped with a smooth measure satisfies $\mbox{MCP}(0,N)$ for some $N>0$. 
\end{theorem}

In the case of Carnot groups which are, as Lie groups equipped with left-invariant sub-Riemannian structures, analytic manifolds with analytic sub-Riemannian structures, the homogeneity allows us to extended the above result to left-invariant Lipschitz distributions. \\

Following \cite{riffordbook}, we say that a sub-Riemannian structure $(\Delta,g)$ or a Carnot group whose first layer $\Delta$ is equipped with a left-invariant metric, is {\it Lipschitz} if it is complete and the associated geodesic distance $d_{SR} : M\times M \rightarrow \R$ is locally Lipschitz outside of the diagonal $D=\{(x,y) \in M\times M\, \vert \, x=y \}$. Examples of Lipschitz sub-Riemannian structures include two-step distributions and more generally medium-fat distributions. A distribution $\Delta$ (or a sub-Riemannian structure with distribution $\Delta$ or a Carnot group whose first layer $\Delta$ is equipped with a left-invariant metric) is called \textit{medium-fat}  if, for every $x\in M$ and every smooth section $X$ of $\Delta$ with $X(x) \neq 0$, there holds
\begin{eqnarray}\label{EQmediumfat}
T_xM = \Delta(x)+ [\Delta,\Delta](x) + \bigl[X,[\Delta,\Delta] \bigr](x),
\end{eqnarray}
where 
$$
 \bigl[X,[\Delta,\Delta] \bigr] (x) := \Bigl\{ \bigl[X,[Y,Z]\bigr] (x) \, \vert \,  Y, Z \mbox{ smooth sections of } \Delta \Bigr\}.
$$
The notion of medium-fat distribution has been introduced by Agrachev and Sarychev in \cite{as99}.    Of course, in the case of a Carnot group the property of being medium-fat depends only on the properties of its Lie algebra. Our second result is the following:

\begin{theorem}\label{THMCarnot}
Any Lipschitz Carnot group whose first layer is equipped with a left-invariant metric and equipped with Haar measure satisfies $\mbox{MCP}(0,N)$ for some $N>0$. 
\end{theorem}

The proofs of Theorem \ref{THMgen} and \ref{THMCarnot} are based on the fact that squared sub-Riemannian pointed distances $d_{SR}(x,\cdot)^2$ satisfy a certain property of horizontal semiconcavity.  Note that for the moment, we are only able to prove this property in the analytic case under an assumption of compactness of length minimizers. The property of horizontal semiconcavity together with the lipschitzness of $d_{SR}(x,\cdot)^2$ allows us to give an upper bound for divergence of horizontal gradients of $f^x$ which implies the desired measure contraction property. \\

It is worth to notice that, thanks to a seminal result by Cavaletti and Huesmann \cite{ch15}, measure contraction properties are strongly connected with the well-posedness of the Monge problem for quadratic geodesic distances. We refer the interested reader to \cite{badreddine17,badreddinethesis} for further details.\\

We recall that all the notations used throughout the paper are listed in Appendix \ref{Notations}. The material required for the proof of the two theorems above  is worked out in Section \ref{SECprel}. The proofs of Theorems \ref{THMgen} and \ref{THMCarnot} are respectively given in Sections \ref{SECproofTHMgen} and \ref{SECproofTHMCarnot}.\\

{\bf Acknowledgement.} The authors are grateful to the referee for useful remarks and for pointing out a gap in the initial proof of Proposition 12. The authors are also indebted to Adam Parusinski for fruitful discussions and the reference to the paper by Denef and Van den Dries \cite{dvdd88}.

\section{Preliminaries}\label{SECprel}

Throughout all this section,  $(\Delta,g)$ denotes a complete sub-Riemannian structure on $M$ of rank $m\leq n$. 

\subsection{The minimizing Sard conjecture}

The minimizing Sard conjecture is concerned with the size of points that can be  reached from a given point by singular minimizing geodesics. Following \cite{riffordbourbaki}, given $x\in M$, we set
$$
\mathcal{S}^x_{\Delta,min^g} := \Bigl\{ \gamma(1) \, \vert \, \gamma \in W^{1,2}_{\Delta}([0,1],M), \gamma \mbox{ sing., }  d_{SR}(x,\gamma(1))^2=\mbox{energy}_{g} (\gamma) \Bigr\}.
$$
Note that for every $x \in M$, the set $\mathcal{S}^x_{\Delta,min^g}$ is closed and contains $x$ (because $m<n$). Let us introduce the following definition.

\begin{definition}\label{DEFSardmin}
We say that $(\Delta,g)$ satisfies the minimizing Sard conjecture at $x\in M$ if the set $\mathcal{S}^x_{\Delta,min^g}$  has Lebesgue measure zero in $M$. We say that it  satisfies the minimizing Sard conjecture if this property holds for any $x\in M$.
\end{definition}

It is not known if all complete sub-Riemannian structures satisfy the minimizing Sard conjecture (see \cite{agrachev14,riffordbourbaki}). The best general result is due to Agrachev who proved in \cite{agrachev09} that all closed sets  $\mathcal{S}^x_{\Delta,min^g}$ have empty interior. As the next result shows, the minimizing Sard conjecture is related to regularity properties of pointed distance functions. Following Agrachev  \cite{agrachev09}, we call smooth point of the function $y\mapsto d_{SR}(x,y)$ (for a fixed $x\in M$) any $y\in M$ for which there is $p\in T_x^*M$ which is not a critical point of the exponential mapping $\exp_x$ and such that the projection $\gamma_{x,p}$ of the normal extremal $\psi:[0,1] \rightarrow T^*M$ starting at $(x,p)$ is the unique minimizing geodesic from $x$ to $y=\gamma_{x,p}(1)$. By Agrachev's Theorem, the set $\mathcal{O}_x$ of smooth points is always open and dense in $M$. The following holds:

\begin{proposition}\label{PROPSardminEQ}
Let $x\in M$ be fixed, the following properties are equivalent:
\begin{itemize}
\item[(i)] the structure $(\Delta,g)$ satisfies the minimizing Sard conjecture at $x\in M$,
\item[(ii)] the function $y\mapsto d_{SR}(x,y)$ is differentiable almost everywhere in $M$,
\item[(iii)] the set of smooth points $\mathcal{O}_x$ is an open set with full measure in $M$. 
\end{itemize}
Furthermore, the function $y\mapsto d_{SR}(x,y)$ is smooth on $\mathcal{O}_x$ and if $M$ and $(\Delta,g)$ are analytic, then the set  $\mathcal{O}_x$ is geodesically star-shaped at $x$, that is 
\begin{eqnarray}\label{starshaped}
\gamma(s,y)\in \mathcal{O}_x \qquad \forall s\in (0,1], \, \forall y \in \mathcal{O}_x, 
\end{eqnarray}
where $\gamma_x (\cdot,y)\in W_{\Delta}^{1,2}([0,1],M)$ is the unique minimizing geodesic from $x$ to $y$. 
\end{proposition}

\begin{proof}[Proof of Proposition \ref{PROPSardminEQ}]
Let $x\in M$ be fixed. The part (iii) $\Rightarrow$ (ii) is immediate. Let us prove that  (ii) $\Rightarrow$ (i).  By assumption the set   of differentiability $D$ of $f:=d_{SR}(x,\cdot)$ has full measure in $M$. Recall that for every $y\in D$, there is a unique minimizing geodesic from $x$ to $y$ which is given by the projection of the normal extremal $\psi:[0,1] \rightarrow T^*M$ such that $\psi(1)=(y,d_{SR}(x,y)d_yf)$ (see \cite[Lemma 2.15 p. 54]{riffordbook}). By Sard's Theorem, the set $S$ of $\exp_x(p)$ with $p\in T_x^*M$ critical has Lebesgue measure zero in $M$. Therefore, the set $D\setminus S$ has full measure and for every $y\in D\setminus S$ there is a unique minimizing geodesic from $x$ to $y$ and it is not singular, which shows that $y$ does not belong to $\mathcal{S}^x_{\Delta,min^g}$. Let us now show that  (i) $\Rightarrow$ (iii). By definition of $\mathcal{S}^x_{\Delta,min^g}$, for every $y\notin \mathcal{S}^x_{\Delta,min^g}$ all minimizing horizontal paths between $x$ and $y$ are not singular. So repeating the proof of \cite[Theorem 3.14 p. 98]{riffordbook} (see also \cite{cr08}), we can show that the function $f:y\mapsto  d_{SR}(x,y)$ is locally semiconcave and so locally Lipschitz on the open set $U:=M\setminus \mathcal{S}^x_{\Delta,min^g}$.  Thus for every compact set $K \subset U$, there is a compact set $\mathcal{P}_K\subset T^*_xM$ such that for every $y\in K$,  there is $p\in \mathcal{P}_K$ with $\exp_x(p)=y$ and $H(x,p)=d_{SR}(x,y)^2/2$ (in other words $\gamma_{x,p}:[0,1]\rightarrow M$ is a  minimizing geodesic from $x$ to $y$). By Sard's Theorem, the set $S_K$ of $\exp_x(p)$ with $p\in \mathcal{P}_K$ critical is a closed set of Lebesgue measure zero. For every positive integer $k$, set (here the diameter of the convex set $d_y^+f$ is taken with respect to some geodesic distance on $T^*M$)
$$
\Sigma^{k} (f) := \Bigl\{y \in U \, \vert \, \mbox{diam} (d^+_yf) \geq 1/k\Bigr\}.
$$
By local semiconcavity of $f$ in $U$, each set $\Sigma^{k} (f)$ is a closed set in $U$ with Lebesgue measure zero (see \cite[Proposition 4.1.3 p. 79]{cs04}). We claim that 
$$
S_K' := K \cap \overline{ \bigcup_{k>0} \Sigma^k(f)} \subset  \left( K \cap \bigcup_{k>0} \Sigma^k(f) \right) \cup S_K.
$$
As a matter of fact, if $y\in K$ belongs to $ \overline{ \cup_{k>0} \Sigma^k(f)} \setminus \cup_{k>0} \Sigma^k(f)$, then $d_y^+f$ is a singleton and there is a sequence $\{y_l\}_l$ converging to $y$ such that all $d_{y_l}^+f$ have dimension at least one and tend to $d_y^+f$. This implies that the covector $p$ such that  $\exp_x(p)=y$ and $H(x,p)=d_{SR}(x,y)^2/2$ is critical, which shows that $y$ belongs to $S_K$. By construction, every point in $K\setminus S_K'$ is a smooth point. We conclude easily.

It remains to prove the second part. The smoothness of $f:y\mapsto  d_{SR}(x,y)$  is an easy consequence of the inverse function theorem. As a matter of fact, we can show easily that for every $y\in \mathcal{O}_x$ such that $y=\exp_x(p)$ with $H(x,p)=d_{SR}(x,y)^2/2$ and $p\in T_x^*M$ non-critical, there is a neighborhood $U$ of $y$ in $\mathcal{O}_x$ such that 
$$
f(z)^2= 2H(x,\exp_x(z)^{-1}) \qquad \forall z \in U,
$$
where $\exp_x^{-1}$ denotes a local inverse of the exponential mapping from a neighborhood of $p$ to $U$. To prove (\ref{starshaped}), we argue by contradiction.  If there are $x\in M$, $y\in \mathcal{O}_x$ and $s\in (0,1)$ such that $z:=\gamma(s,y)\in \mathcal{O}_x$ then either there are two distinct minimizing geodesics from $x$ to $z$ or there is only one minimizing geodesic from $x$ to $y$ which is singular. In the first case, we infer the existence of two distinct minimizing geodesics from $x$ to $y$, which contradicts the smoothness of $y$. In the second case, we deduce that the minimizing geodesic $\gamma_x(\cdot,y)$ is the projection of a normal extremal which is regular and whose restriction to $[0,s]$ is singular. This cannot happen under the assumption of analyticity of the datas because an abnormal extremal above  $\gamma_x(\cdot,y)$ over $[0,s]$ could be extended to an abnormal extremal over $[0,1]$ (see \cite[Proposition 1.11 p.21]{riffordbook}).
\end{proof}

\begin{remark}
By Proposition \ref{PROPSardminEQ}, any (complete) sub-Riemannian structure satisfying the minimizing Sard conjecture has negligeable cut loci. 
\end{remark}


\begin{remark}\label{REMample}
As pointed out by the referee, the set $\mathcal{O}_x$ could be replaced by the set $\mathcal{A}_x\subset \mathcal{O}_x$ of ample points from $x$. This set is defined as the set of $y \in \mathcal{O}_x$ where the unique minimizing geodesic from $x$ to $y$ is ample, that is whose growth vector saturates the tangent space. It can be shown to be open with full measure and geodesically star-shaped in the smooth case, we refer the interested reader to the monograph \cite{abr18}  for further details. In our case, since we need the analyticity to prove other results, we prefer to work with the simpler $\mathcal{O}_x$ which is geodesically star-shaped in the analytic case. 
\end{remark}

\subsection{Two characterizations for $\mbox{MCP}(0,N)$}\label{SECMCP}

The following result was implicit in the previous paper  \cite{rifford13} of the second author (it is also the case in \cite[Page 5]{lee15} and \cite[Section 6.2]{ohta14}). The measure contraction property $\mbox{MCP}(0,N)$  is equivalent to some upper bound on the divergence of the horizontal gradient of the squared pointed sub-Riemannian distance. This result holds at least whenever the horizontal gradient is well-defined and  the sets $\mathcal{O}_x$ are geodesically star-shaped.

\begin{proposition}\label{PROPversus}
Assume that $(\Delta,g)$ satisfies the minimizing Sard conjecture and that all its sets $\mathcal{O}_x$ are geodesically star-shaped, and let $\mu$ be a smooth measure on $M$ and $N>0$ be fixed. Then  $(\Delta,g)$ equipped with $\mu$ satisfies $\mbox{MCP}(0,N)$ if and only if 
\begin{eqnarray}\label{divfxN}
\mbox{div}_y^{\, \mu} \left( \nabla^h f^x \right)  \leq N \qquad \forall y \in \mathcal{O}_x, \, \forall x \in M,
\end{eqnarray}
where $f^x:M\rightarrow \R$ is the function defined by $f^x(y):=d_{SR}(x,y)^2/2$. 
\end{proposition}

\begin{proof}
Let $x\in M$ be fixed, the vector field $Z:=-\nabla^h f^x $ is well-defined and smooth on $\mathcal{O}_x$. Moreover by assumption, every solution of $\dot{y}(t)=Z(y(t))$ with $y(0)\in \mathcal{O}_x$ remains in $\mathcal{O}_x$ for all $t\geq 0$, we denote by $\{\varphi_t\}_{t\geq0}$ the flow of  $Z$ on $\mathcal{O}_x$. For every $y\in \mathcal{O}_x$, the function $\theta : t\in [0,+\infty) \mapsto d_{SR}(\varphi_t(y),y)$ satisfies  
$$
\theta(0)=0 \quad \mbox{and} \quad \theta(t) = \mbox{length}^g \left( \varphi_{[0,t]}(y)\right) = \int_0^t \left| Z(\varphi_s(y)  \right|\, ds.
$$
So that, for all $t\geq 0$,
$$
\dot{\theta}(t)= \left| Z(\varphi_t(y)  \right| = d_{SR} \left(x,\varphi_t(y)\right) = d_{SR}(x,y) - d_{SR}\left(y,\varphi_t(y)\right) = d_{SR}(x,y) -  \theta(t),
$$
which yields 
$$
\theta(t) = d_{SR}(x,y) \left( 1-e^{-t}\right) \qquad \forall t \geq 0.
$$
Consequently, if $A\subset \mathcal{O}_x$ is a Borel set and $s\in (0,1]$, then we have
$$
A_s = \Bigl\{ \gamma_x(s,y) \, \vert \, y \in A\Bigr\} = \varphi_t(A) \quad \mbox{with} \quad t=- \ln (s).
$$
Let us now assume that (\ref{divfxN}) is satisfied. By definition of $\mbox{div}^{\mu}Z$, for every $x\in M$ and any measurable set $A\subset \mathcal{O}_x$, we have for every $t\geq 0$ (see for example, see \cite[Proposition B.1]{br16}),
$$
\mu\left(\varphi_t(A) \right) = \int_A \exp \left( \int_0^t \mbox{div}_{\varphi_s(y)}^{\mu} (Z) \, ds\right) \, d\mu(y),
$$
which by (\ref{divfxN}) implies with $s=e^{-t}$,
$$
\mu \left( A_s\right) = \mu\left(\varphi_t(A) \right) \geq   \int_A \exp \left(-Nt\right) \, d\mu(y) =  s^N \mu(A).
$$
This shows that  (\ref{divfxN}) implies $\mbox{MCP}(0,N)$.  Conversely, if  $(\Delta,g)$ equipped with $\mu$ satisfies $\mbox{MCP}(0,N)$ then for every $x\in M$ and every small ball $B_{\delta}(y)\subset \mathcal{O}_x$ (say a Riemannian ball with respect to the Riemannian extension $g$), we have 
$$
\mu\left(\varphi_t \left(B_{\delta}(y)\right) \right)   =  \int_{B_{\delta}(y)} \exp \left( \int_0^t \mbox{div}_{\varphi_s(y)}^{\mu} (Z) \, ds\right) \, d\mu(y)\geq  e^{-Nt}\,  \mu \left(B_{\delta} (y) \right) \qquad \forall t \geq 0.
$$
For every $t\geq 0$, letting $\delta$ go to $0$ yields 
$$
\exp \left( \int_0^t \mbox{div}_{\varphi_s(y)}^{\mu} (Z) \, ds\right)\geq  e^{-Nt}.
$$
We infer (\ref{divfxN}) by dividing by $t$ and letting $t$ go to $0$.
\end{proof}

In the case of Carnot groups, the invariance of the divergence of $\nabla^h f^x$ by dilation allows us to characterize $\mbox{MCP}(0,N)$ in term of a control on the divergence over a compact set not containing the origin.

\begin{proposition}\label{PROPCarnoteq}
Let $\G$ be a Carnot group whose first layer is equipped with a left-invariant metric satisfying the minimizing Sard conjecture and $N>0$ fixed. Then the metric space $(\G,d_{SR})$ with Haar measure $\mu$ satisfies $\mbox{MCP}(0,N)$ if and only if  
\begin{eqnarray}\label{divfxNsph}
\mbox{div}_y^{\, \mu} \left( \nabla^h f^0 \right)  \leq N \qquad \forall y \in \mathcal{O}_0 \cap  S_{SR}(0,1),
\end{eqnarray}
where $f^0:M\rightarrow \R$ is the function defined by $f^0(y):=d_{SR}(0,y)^2/2$. 
\end{proposition}

\begin{proof}
Since Carnot groups are indeed analytic, by the second part of Proposition \ref{PROPSardminEQ} and Proposition \ref{PROPversus}, it is sufficient to show that (\ref{divfxNsph}) is equivalent to
\begin{eqnarray}\label{21dec1}
\mbox{div}_y^{\, \mu} \left( \nabla^h f^0 \right)  \leq N \qquad \forall y \in \mathcal{O}_0.
\end{eqnarray}
Recall that by taking a set of exponential coordinates $(x_1, \ldots,x_n)$, we can identify $\G$ with its Lie algebra $\mathfrak{g}\simeq \R^n$ and indeed consider that we work with the Lebesgue measure in $\R^n$ and that the sub-Riemannian structure is globally parametrized by an orthonormal family of analytic vector fields $X^1, \ldots, X^n$ in $\R^n$ satisfying
\begin{eqnarray}\label{homo}
X^i \left( \delta_{\lambda}(x)\right) = \lambda^{-1} \, \delta_{\lambda} \left( X^i(x)\right) \qquad \forall x \in \R^n, \, \forall i=1, \ldots,n,
\end{eqnarray}
where $\{\delta_{\lambda}\}_{\lambda > 0}$ is a family of dilations defined as ($d_1, \ldots, d_n$ are positive integers)
$$
\delta_{\lambda} \left(x_1,\ldots, x_n\right) = \left( \lambda^{d_1} x_1,  \lambda^{d_2} x_2, \ldots, \lambda^{d_n} x_n \right)  \qquad \forall x \in \R^n.
$$ 
By the homogeneity property, we have $d_{SR} \left(0,\delta_{\lambda}(x)\right) = \lambda \, d_{SR}(0,x)$ for all  $x \in \R^n$  and $\lambda >0$. Then we have 
\begin{eqnarray}\label{homo2}
f^0\left( \delta_{\lambda}(x)\right) = \lambda^2 \, f^0(x) \quad \mbox{and}  \quad d_{\delta_{\lambda}(x)}f^0\circ \delta_{\lambda}= \lambda^2 \, d_xf^0 \qquad \forall x \in \R^n, \, \forall \lambda >0.
\end{eqnarray}
Recall that the horizontal gradient $\nabla^h f^0$ is given by
$$
\nabla^h_x f^0 = \sum_{i=1}^m \left( X^i \cdot f^0\right)(x) X^i(x) \qquad \forall x \in \R^n.
$$
Therefore, by (\ref{homo})-(\ref{homo2}), we infer that for every $x\in \R^n$ and $\lambda >0$,
\begin{eqnarray*}
\nabla^h_{\delta_{\lambda}(x)} f^0  &= & \sum_{i=1}^m d_{\delta_{\lambda}(x)} f^0 \left( X^i \left( \delta_{\lambda}(x)\right) \right)  \, X^i \left( \delta_{\lambda}(x)\right) \\
& = & \lambda^{-2} \sum_{i=1}^m d_{\delta_{\lambda}(x)} f^0 \left(  \delta_{\lambda} \left( X^i (x)  \right) \right) \,  \delta_{\lambda} \left( X^i (x)\right)\\
& = & \sum_{i=1}^m d_x f^0  \left( X^i (x)  \right)  \,  \delta_{\lambda} \left( X^i (x)\right) =  \delta_{\lambda} \left( \nabla^h_x f^0\right).
\end{eqnarray*}
We deduce that 
$$
\mbox{div}_{\delta_{\lambda}(x)}^{\, \mu} \left( \nabla^h f^0 \right) = \mbox{div}_x^{\, \mu} \left( \nabla^h f^0 \right) \qquad \forall x \in \R^n, \, \forall \lambda >0,
$$
which shows that  (\ref{divfxNsph})  and (\ref{21dec1}) are equivalent and concludes the proof.
\end{proof}

\subsection{Nearly horizontally semiconcave functions}\label{SECnhscf}

Recall that without loss of generality, we can assume that the metric $g$ over $\Delta$ is the restriction of a global Riemannian metric on $M$. This metric allows us to define the $C^2$-norms of functions from $\R^m$ to $M$. In the following statement, $(e_1, \ldots, e_m)$ stands for the canonical basis in $\R^m$.

\begin{definition}
Let $C>0$ and $U$ an open subset of $M$, a function $f:U \rightarrow \R$ is said to be $C$-nearly horizontally semiconcave with respect to $(\Delta,g)$ if  for every $y\in U$, there are an open neighborhood $V^y$ of $0$ in $\R^m$, a function $\varphi^y: V^y \subset \R^m \rightarrow U$ of class $C^2$  and a function $\psi^y : V^y \subset \R^m \rightarrow \R$ of class $C^2$ such that 
\begin{eqnarray}\label{DEFnhscf1}
\varphi^y(0)=y, \quad \psi^y(0)= f(y), \quad f\left(\varphi^y(v)\right) \leq \psi^y(v) \, \, \, \forall v \in V^y, 
\end{eqnarray}
\begin{eqnarray}\label{DEFnhscf2}
\Bigl\{ d_0\varphi^y (e_1), \ldots, d_0\varphi^y(e_m)\Bigr\} \mbox{ is an orthonormal family of vectors in } \Delta(y),
\end{eqnarray}
and
\begin{eqnarray}\label{DEFnhscf3}
\left\|\varphi^y \right\|_{C^2}, \, \left\|\psi^y \right\|_{C^2} \leq C,
\end{eqnarray}
where $\left\|\varphi^y \right\|_{C^2}, \, \left\|\psi^y \right\|_{C^2}$ denote the $C^2$-norms of $\varphi^y$ and $\psi^y$.
\end{definition}

\begin{remark}
We refer the reader to \cite{mm16} for an other notion of horizontal semiconcavity of interest that has been investigated by Montanari and Morbidelli in the framework of Carnot groups.
\end{remark}
 
If $m$ were equal to $n$ that is if we were in the Riemannian case, the above definition would coincide with the classical definitions of semiconcave functions (see \cite{cs04,riffordbook}). Here, in the case $m<n$, the definition tells that at each point, there is a support function from above of class $C^2$ which bounds the function along a $C^2$ submanifold which is tangent to the distribution. This type of mild horizontal semiconcavity will allows us, at least in certain cases, to bound the divergence of the horizontal gradient of squared pointed sub-Riemannian distance functions. 

Before stating the main result of this section, we recall that a minimizing geodesic $\gamma : [0,1] \rightarrow M$ from $x$ to $y$ is called normal if it is the projection of a normal extremal, that is a trajectory $\psi :[0,1] \rightarrow T^*M$ of the Hamiltonian vector field associated with $(\Delta,g)$. We refer the reader to Appendix \ref{Notations} for more details. Here is our result ($|p|^*=|p|_x^*$ for every $(x,p)\in T^*M$ is the dual norm associated with $g$):

\begin{proposition}\label{PROPhsc}
Assume that $M$ and  $(\Delta,g)$ are analytic and let $K$ be a compact subset of $M$, $U\subset M$ a relatively compact open set of $M$, and $A>0$ satisfying the following property: For every $x\in K$, every $y \in \bar{U}$ and every minimizing geodesic $\gamma:[0,1] \rightarrow M$ from $x$ to $y$, there is $p\in T_x^*M$ with $|p|^*\leq A$ such that $\gamma$ is the projection of the normal extremal $\psi:[0,1] \rightarrow T^*M$ starting at $(x,p)$. Then there is $C>0$ such that for every $x\in K$, the function $y\mapsto d_{SR}(x,y)^2$ is $C$-nearly horizontally semiconcave in $U$.
\end{proposition}

\begin{proof}[Proof of Proposition \ref{PROPhsc}]
Let $K$ be a compact set of $M$, $U$ be a relatively compact open set of $M$ and $\bar{x} \in K$ fixed, let us first show how to construct functions $\varphi^{\bar{y}}, \psi^{\bar{y}}$ of class $C^2$ satisfying (\ref{DEFnhscf1})-(\ref{DEFnhscf2}) for  $\bar{y} \in \bar{U}$. Pick a minimizing geodesic $\bar{\gamma}:[0,1] \rightarrow M$ from $\bar{x}$ to $\bar{y}=\bar{\gamma}(1)$. There is an open neighborhood $U_{\bar{\gamma}}$ of $\bar{\gamma}([0,1])$ and a family $\mathcal{F}_{\bar{\gamma}}$ of $m$ analytic vector fields $X^1_{\bar{\gamma}}, \ldots, X_{\bar{\gamma}}^m$ in $M$ such that for every $z\in U_{\bar{\gamma}}$ the family $\{X_{\bar{\gamma}}^1(z), \ldots, X_{\bar{\gamma}}^m(z)\}$ is orthonormal with respect to $g$ and parametrize $\Delta$ (that is $\mbox{Span} \{X_{\bar{\gamma}}^1(z), \ldots, X_{\bar{\gamma}}^m(z)\}=\Delta(z)$) and for every $z \in M\setminus U_{\bar{\gamma}}$, $X_{\bar{\gamma}}^1(z), \ldots, X_{\bar{\gamma}}^m(z)$ belongs to $\Delta(z)$. Consider the End-Point mapping from $\bar{x}$ in time $1$ associated with the family $\mathcal{F}_{\bar{\gamma}}=\{X^1_{\bar{\gamma}}, \ldots, X_{\bar{\gamma}}^m\}$, it is defined by 
$$
E^{\bar{x},1}_{\mathcal{F}_{\bar{\gamma}}}: u\in L^{2}([0,1];\R^m) \, \longmapsto \, \gamma_u^{\mathcal{F}_{\bar{\gamma}}}(1)\in M,
$$
where $\gamma_u^{\mathcal{F}_{\bar{\gamma}}}(1) :[0,1] \rightarrow M$ is the solution to the Cauchy problem
\begin{eqnarray}\label{CPproof}
\dot{\gamma} (t) = \sum_{i=1}^m u_i(t) X_{\bar{\gamma}}^i(\gamma(t)) \, \mbox{ for a.e. } t \in [0,1], \quad \gamma(0)= \bar{x}.
\end{eqnarray}
Note that taking the vector fields $X_{\bar{\gamma}}^1, \ldots, X_{\bar{\gamma}}^m$ equal to zero outside of an neighborhood of $U_{\bar{\gamma}}$, we may assume without loss of generality that $E^{x,1}_{\mathcal{F}_{\bar{\gamma}}}$ is well-defined on $L^{2}([0,1];\R^m)$. Recall that the function $E^{x,1}_{\mathcal{F}_{\bar{\gamma}}}$ is smooth and satisfies (see \cite[Proposition 1.10 p. 19]{riffordbook})
\begin{eqnarray}\label{benin}
\Delta \left( \bar{y} \right) \subset \mbox{Im} \left(d_{u_{\bar{\gamma}}}E_{\mathcal{F}_{\bar{\gamma}}}^{\bar{x},1}\right),
\end{eqnarray}
where $u_{\bar{\gamma}}$ is the unique control $u\in L^2([0,1],\R^m)$ such that $\gamma_u^{\mathcal{F}_{\bar{\gamma}}}=\bar{\gamma}$. Therefore, there are  $v_{\bar{\gamma}}^1, \ldots, v_{\bar{\gamma}}^m \in L^{2}([0,1],\R^m)$ such that 
\begin{eqnarray}\label{3-12Proof1}
 d_{u_{\bar{\gamma}}}E^{\bar{x},1}_{\mathcal{F}_{\bar{\gamma}}}\left(v^i_{\bar{\gamma}}\right) = X_{\bar{\gamma}}^i \left( \bar{y} \right) \qquad \forall i=1, \ldots,m.
\end{eqnarray}
Define $\varphi_{\bar{\gamma}}:\R^m \rightarrow M$ by 
$$
\varphi_{\bar{\gamma}} (v) := E^{\bar{x},1}_{\mathcal{F}_{\bar{\gamma}}} \left( u_{\bar{\gamma}} + \sum_{i=1}^m v_i v_{\bar{\gamma}}^i\right) \qquad \forall v = (v_1, \ldots, v_m) \in \R^m.
$$
By construction, $\varphi_{\bar{\gamma}}$ is smooth and satisfies
$$
\varphi_{\bar{\gamma}}(0)=E^{\bar{x},1}_{\mathcal{F}_{\bar{\gamma}}} \left( u_{\bar{\gamma}} \right)= \bar{\gamma}(1)= \bar{y},
$$
and
$$
 d_0 \varphi_{\bar{\gamma}}(e_i) = d_{u_{\bar{\gamma}}}E^{ \bar{x},1}_{\mathcal{F}_{\bar{\gamma}}} \left(v^i_{\bar{\gamma}}\right)=X^i_{\bar{\gamma}} \left( \bar{y} \right) \qquad \forall i=1, \ldots,m.
$$
Moreover, for every $v \in \R^m$ such that the solution to (\ref{CPproof}) associated with the control $u_{\bar{\gamma}} + \sum_{i=1}^m v_i v_{\bar{\gamma}}^i$ remains in $U_{\bar{\gamma}}$, we have 
$$
d_{SR} \left( \bar{x},\varphi_{\bar{\gamma}}(v) \right)^2 \leq      \left\|  u_{\bar{\gamma}} + \sum_{i=1}^m v_i v_{\bar{\gamma}}^i\right\|_{L^2}^2 =:\psi_{\bar{\gamma}}(v).
$$
By construction, $\varphi_{\bar{\gamma}}$ and $\psi_{\bar{\gamma}}$ are smooth, defined in a neighborhood of $0\in \R^m$ and satisfy (\ref{DEFnhscf1})-(\ref{DEFnhscf2}). It remains to show that the $C^2$ norms of $\varphi^{\bar{\gamma}}$ and $\psi^{\bar{\gamma}}$ can be taken to be uniformly bounded, this is the purpose of the next lemma.

\begin{lemma}\label{LEMaims}
There are neighborhoods $U_{\bar{x}}$ and $U_{\bar{y}}$ respectively of $\bar{x}$ and $\bar{y}$ in $M$, a neighborhood $\mathcal{U}_{\bar{\gamma}}$  of $u_{\bar{\gamma}}$ in $L^2([0,1],\R^m)$  and $C_{\bar{\gamma}}>0$ such that for every $x \in K\cap U_{\bar{x}}$, $y\in \bar{U} \cap U_{\bar{y}}$ and every control $u_{x,y}$ associated with a minimizing geodesic $\gamma_{x,y}:[0,1] \rightarrow M$ from $x$ to $y$ with $u_{x,y} \in \mathcal{U}_{\bar{\gamma}}$, there are $v^1_{u_{x,y}}, \ldots, v^m_{u_{x,y}} \in L^{2}([0,1],\R^m)$ such that 
\begin{eqnarray}\label{aims1}
d_{u_{x,y}}E^{x,1}_{\mathcal{F}_{\bar{\gamma}}}\left(v^i_{u_{x,y}}\right) = X_{\bar{\gamma}}^i  (y) \qquad \forall i=1, \ldots,m
\end{eqnarray}
and
\begin{eqnarray}\label{aims2}
\left\| v_{u_{x,y}}^i\right\|_{L^2} \leq C_{\bar{\gamma}}  \qquad \forall i=1, \ldots,m.
\end{eqnarray}
\end{lemma}
\begin{proof}[Proof of Lemma \ref{LEMaims}]
Note that if we prove for each $i\in \{1, \ldots, m\}$  the existence of neighborhoods $U_x^i, U_y^i$ and $\mathcal{U}_{\bar{\gamma}}^i$ such that (\ref{aims1})-(\ref{aims2}) are satisfied for $i$, then the result follows by taking the intersections of the neighborhoods  $U_{\bar{x}}^i, U_{\bar{y}}^i$ and $\mathcal{U}_{\bar{\gamma}}^i$. So, let us fix $i$ in $\{1, \ldots,m\}$. By taking a chart on a neighborhood of $\bar{\gamma}([0,1])$ (that we still denote by $U_{\bar{\gamma}}$)  we may assume that the restriction of $E^{\bar{x},1}_{\mathcal{F}_{\bar{\gamma}}}$ to a neighborhood of $u_{\bar{\gamma}}$ is valued in $\R^n$ and doing a change of coordinates $(y_1, \ldots, y_n)$  in a neighborhood $U_{\bar{y}}^0$ we may also assume that there are analytic vector fields $Y^1,\ldots, Y^m$ such that 
$$
Y^1(y)= X^i (y) = \partial_{y_1}, \quad  Y^j (y)= \partial_{y_j} + \sum_{l=m+1}^n a^j_l (y) \, \partial_{y_l} \qquad \forall j=2, \ldots,m, \, \forall y \in U_{\bar{y}}^0
$$
and
$$
\Delta (z) = \mbox{Span} \Bigl\{ Y^1(z),\ldots, Y^m (z)\Bigr\} \qquad \forall z \in U_{\bar{\gamma}}.
$$
Observe that by construction, there are analytic mappings $f^j_k$ on $U_{\bar{\gamma}}$  for $j\in \{1, \ldots,m\}$ and $k\in \{1, \ldots, m\}$ such that 
$$
Y^j(z) = \sum_{k=1}^m f^j_k(z) \, X^k(z) \qquad \forall z \in U_{\bar{\gamma}}, \, \forall j =1, \ldots, m.
$$
As a consequence, if $\gamma :[0,1] \rightarrow U_{\bar{\gamma}}$ is solution to 
\begin{eqnarray}\label{aimsday2}
\dot{\gamma}(t) =  \sum_{j=1}^m w_j(t) \, Y^j(\gamma(t)) \, \mbox{ for a.e. } t \in [0,1],
\end{eqnarray}
for some $w\in L^2([0,1],\R^m)$, then there holds for a.e. $t\in [0,1]$,
\begin{eqnarray*}
\dot{\gamma}(t)  =     \sum_{j=1}^m w_j(t) \, \sum_{k=1}^m f_k^j(\gamma(t)) \, X^j(\gamma(t)) = \sum_{k=1}^m \left( \sum_{j=1}^m w_j(t) \, f_k^j(\gamma(t)) \right) \, X^k(\gamma(t)).
\end{eqnarray*}
This shows that if we denote by $\mathcal{F}$ the family $\{Y^1, \ldots, Y^m\}$, then any horizontal curve parametrized by $\mathcal{F}$ can be parametrized by $\mathcal{F}_{\bar{\gamma}}$ as above. On the other hand, any parametrization with respect to $\mathcal{F}_{\bar{\gamma}}$ leads to a parametrization with respect to $\mathcal{F}$.  Consequently, there is a control $\bar{w}\in L^2([0,1),\R^m)$ such that the solution of (\ref{aimsday2}) starting at $\bar{x}$ is equal to $\bar{\gamma}$ and there are neighborhoods $U_{\bar{x}}^0$ of $\bar{x}$ and  $\mathcal{W}$ of $\bar{w}$ together with a mapping  $G: U_{\bar{x}}^0 \times \mathcal{W} \rightarrow L^2([0,1],\R^m)$ of class $C^1$ such that 
$$
E^{x,1}_{\mathcal{F}}(w) = E^{x,1}_{\mathcal{F}_{\bar{\gamma}}} ( G(x,w)) \qquad \forall x\in U_{\bar{x}}^0, \, \forall w \in \mathcal{W},
$$
where $E^{x,1}_{\mathcal{F}}$ denotes the End-Point mapping from $x$ in time $1$ associated with the family $\mathcal{F}$. Therefore, in order to prove our result it is sufficient to prove that there is $\bar{C}>0$ such that for every $x$ in $K$ close to $\bar{x}$, every $y$ in $\bar{U}$ close to $\bar{y}$ and any $\gamma$ minimizing geodesic from $x$ to $y$ close to $\bar{\gamma}$ associated to the control $w\in \mathcal{W}$ with respect to $\mathcal{F}$, there is  $\omega\in L^2([0,1],\R^m)$ such that 
\begin{eqnarray}\label{aimsday2_1}
d_wE^{x,1}_{\mathcal{F}} (\omega) = X^i (y) = \partial_{y_1} \quad \mbox{and} \quad \left\| \omega\right\|_{L^2} \leq \bar{C}.
\end{eqnarray}

Let $x\in U_{\bar{x}}^0$ and $w\in \mathcal{W}$ be fixed, the differential of $E^{x,1}_{\mathcal{F}}$ at $w$ is given by (see \cite[Remark 1.5 p.15]{riffordbook})
$$
d_{w}E^{x,1}_{\mathcal{F}}(v) = S(1) \int_0^1 S(t)^{-1} B(t) v(t) \, dt \qquad \forall v \in L^2([0,1],\R^m),
$$
where $S:[0,1] \rightarrow M_n(\R)$ is solution to the Cauchy problem
$$
\dot{S}(t) =A(t) S(t) \quad \mbox{for a.e. } t \in [0,1], \quad S(0)=I_n
$$
and where the matrices $A(t)\in M_n(\R), B(t) \in M_{n,m}(\R)$ are defined by (note that $J_{Y^{1}}=J_{X^i}=0$)
$$
A(t) := \sum_{j=2}^m w_j(t) \, J_{Y^j} \left( \gamma(t) \right) \qquad \mbox{a.e. } t \in [0,1],
$$
$$
B(t) = \left( Y^1 \left( \gamma(t) \right),  \ldots, Y^m \left( \gamma (t) \right) \right) \qquad \forall t \in [0,1].
$$
By construction of $Y^1, \ldots, Y^m$, there is $\tau^0 \in (0,1)$ (which does not depend upon $x, y$ but only upon $\bar{x}, \bar{y}$) such that for almost every $t\in [1-\tau^0,1]$ the matrix $A(t)$ has the form
$$
A(t) = 
\begin{pmatrix} 
0 & \vline & \begin{matrix} 0 & \cdots &  \cdots & 0 \end{matrix} &\vline &  \begin{matrix} 0 & \cdots & 0 \end{matrix} \\
\hline 
\begin{matrix} 
0 \\
\vdots \\
\vdots \\
0
\end{matrix}
& \vline & 
\bigzero_{m-1,m-1}
& \vline & \bigzero_{m-1,n-m} \\
\hline
\alpha(t)
& \vline & \beta(t) & \vline & \delta(t)
\end{pmatrix},
$$
where $\alpha(t)$ is in $\R^{n-m}$, $\beta(t)=(\beta(t)_{k,l})$ is a $(n-m)\times (m-1)$ matrix and $\delta(t)$ is a $(n-m)$ square matrix.  
Therefore, as the solution of the Cauchy problem $\dot{\tilde{S}}(t)=\tilde{S}(t) A(1-t), \tilde{S}(0)=I_n$, the first column of the matrix $\tilde{S}(t):=S(1)S(1-t)^{-1}$ with $t\in [0,1]$ has the form
\begin{eqnarray}\label{aimsday2_3}
\begin{pmatrix}
\tilde{s}_1(t) \\
\vdots \\
\tilde{s}_n(t)
\end{pmatrix}
\quad \mbox{with} \quad
\left\{
\begin{array}{l}
\tilde{s}_1(t) = 1 \\
\tilde{s}_j(t) = 0 \quad \forall j=2, \ldots, m
\end{array}
\right.
\quad \forall t \in [0,\tau_0]
\end{eqnarray}
and the column vector $\tilde{s}(t)$ with coordinates $(\tilde{s}_{m+1}(t), \ldots, \tilde{s}_n(t))$ is  solution of 
\begin{eqnarray}\label{aimsday2_2}
\dot{\tilde{s}}(t) = D(t) \, \alpha(1-t) \quad \forall t \in [0,\tau_0], \quad \tilde{s}(0)=0
\end{eqnarray}
with $D(t)$ the $(n-m)$ square matrices satisfying  
\begin{eqnarray}\label{aimsday2_22}
 \dot{D}(t) = D(t) \delta(1-t) \quad \forall t \in [0,\tau_0], \quad  D(0)=I_{n-m}.
\end{eqnarray}
Thus, a way to solve $d_w E^{x,1}_{\mathcal{F}}(\omega)=\partial_{y_1}$ is to take $\omega \in L^2([0,1], \R^m)$ of the form
\begin{eqnarray}\label{johan1}
\omega(t) = \left(\omega_1(1-\cdot),0, \ldots, 0\right)\qquad \mbox{a.e. } t \in [0,1]
\end{eqnarray}
 with 
\begin{multline}\label{johan2}
\mbox{Supp}(\omega_1) \subset [0,\tau^0], \quad \int_0^1 \omega_1(t) \, dt = 1 \\
\quad \mbox{and} \quad \int_0^1 \omega_1(t) \, \tilde{s}_l(t)\, dt = 0 \quad \forall l=m+1, \ldots, n.
\end{multline}

Remember now that by assumption, $M$ and $(\Delta,g)$ are analytic and for every $x\in K$ and $ y \in \bar{U}$ all minimizing geodesics joining $x$ to $y$  are normal with initial covector bounded by $A$. Let $\bar{p}\in T_{\bar{x}}^*M$ be such that $\bar{\gamma}=\gamma_{\bar{x},\bar{p}}$ where for any $(x,p)\in T^*M$, $\gamma_{x,p}:[0,1] \rightarrow M$ denotes the projection of the normal extremal $\psi_{x,p}:[0,1] \rightarrow T^*M$ starting at $(x,p)$ (see Appendix \ref{Notations}). Then, the desired result will follow if we show that there are a neighborhood $\mathcal{T}$ of $(\bar{x},\bar{p})$ in $T^*M$ and $\bar{C}>0$ such that for every $(x,p)\in \mathcal{T}$, the point $x$ belongs to $U_{\bar{x}}^0$, the curve $\gamma_{x,p}([0,1])$ is contained in $U_{\bar{\gamma}}$, the associated control $w=w^{x,p}$ belongs to $\mathcal{W}$ and there is $\omega= \omega^{x,p} \in L^2([0,1],\R^m)$ satisfying (\ref{johan1})-(\ref{johan2}) (with the function $\tilde{s}=\tilde{s}^{x,p}=(\tilde{s}_{m+1}(t), \ldots, \tilde{s}_n(t)) :[0,1]\rightarrow \R^{n-m}$ given by the first column of $\tilde{S}(t)=\tilde{S}^{x,p}(t):=S(1)S(1-t)^{-1}$ associated  with $\gamma_{x,p}$ and $w_{x,p}$ which satisfies (\ref{aimsday2_3})-(\ref{aimsday2_2}) with $\alpha=\alpha^{x,p}$ and $\delta=\delta^{x,p}$ associated with $\gamma_{x,p}$ and $w_{x,p}$ as well) such that   
\begin{eqnarray}\label{johan3}
\left\| \omega^{x,p} \right\|_{L^2} \leq \bar{C}.
\end{eqnarray}
We need the following lemma whose proof is postponed to Appendix \ref{PROOFbenin}.

\begin{lemma}\label{benin4j}
Let $\mathcal{K} \subset \R^l$ be a compact set and $h:[0,1] \times \mathcal{K} \rightarrow \R$ be an analytic mapping such that $h(0,\kappa)=0$ for all $\kappa\in \mathcal{K}$. Then there are $\tau>0$ as small as desired and $\nu \in (0,1)$ such that 
\begin{eqnarray}\label{eq15oct}
\left( \int_0^{\tau} h(t,p) \, dt \right)^2 \leq \nu \, \tau \, \int_0^{\tau} h(t,p)^2 \, dt \qquad \forall \kappa \in \mathcal{K}.
\end{eqnarray}
\end{lemma}

Let $\mathcal{T}$ be a compact neighborhood of  of $(\bar{x},\bar{p})$ in $T^*M$ such that for every $(x,p)\in \mathcal{T}$, the point $x$ belongs to $U_{\bar{x}}^0$, the curve $\gamma_{x,p}([0,1])$ is contained in $U_{\bar{\gamma}}$ and the associated control $w=w^{x,p}$ belongs to $\mathcal{W}$. Then we note that the mapping 
$$
(t,x,p,\lambda) \in [0,1] \times \mathcal{T} \times [-1,1]^{n-m}  \longmapsto h(t,x,p,\lambda) := \sum_{k=1}^{n-m} \lambda_k \tilde{s}_k^{x,p}(t)
$$
 is analytic. Therefore, by Lemma \ref{benin4j}, there are $\tau\in(0,\tau^0)$ and $\nu\in (0,1)$ such that
\begin{eqnarray}\label{johan6}
\left( \int_0^{\tau}  h(t,x,p,\lambda) \, dt \right)^2 \leq \nu \tau \, \int_0^{\tau} h(t,x,p,\lambda)^2 \, dt \qquad \forall (x,p) \in \mathcal{T}, \, \forall \lambda \in  [-1,1]^{n-m}.
\end{eqnarray}
Define $I\in L^2([0,1],\R)$ by $I(t):= 1_{[0,\tau]}(t)$ for all $t\in [0,1]$ and for every $(x,p) \in \mathcal{T}$ denote by $P\in L^2([0,1],\R)$ the orthogonal projection of $I$ in $L^2([0,1],\R)$ over the vector space
$$
V^{x,p}:= \Bigl\{ f \in L^2([0,1],\R) \, \vert \, \langle f, 1_{[0,\tau]} \, \tilde{s}_l^{x,p} \rangle_{L^2}=0, \, \forall l=m+1, \ldots,n \Bigr\}.
$$
Let $(x,p)\in \mathcal{T}$ be fixed, then there are $\Lambda_1^{x,p}, \ldots, \Lambda_{n-m}^{x,p} \in \R$ such that 
$$
P= I + \sum_{k=1}^{n-m} \Lambda_k \, 1_{[0,\tau]} \,  \tilde{s}^{x,p}_{m+k}  =: I+ J
$$
and there holds 
\begin{multline*}
\langle P, 1_{[0,\tau]} \, \tilde{s}^{x,p}_l \rangle_{L^2} = \int_0^1  P(t) 1_{[0,\tau]} (t) \, \tilde{s}_l^{x,p}(t) \, dt =\\
0 =  \int_0^{\tau}  \left[ 1_{[0,\tau]} (t) P(t) \right] \, \tilde{s}_l^{x,p}(t) \, dt  \qquad \forall l=m+1, \ldots, n-m
\end{multline*}
and
\begin{eqnarray*}
 \int_0^1    1_{[0,\tau]}  (t) P(t) (t) \, dt =\langle P,I\rangle_{L^2} =  \|P\|^2_{L^2} & = & \|I\|_{L^2}^2  + \|J\|_{L^2}^2 + 2 \langle I,J\rangle_{L^2} \\
 & \geq &  \|I\|_{L^2}^2  + \|J\|_{L^2}^2 - 2 \sqrt{\nu}  \|I\|_{L^2}  \, \|J\|_{L^2}\\
 & \geq & \left(1- \sqrt{\nu}\right)  \|I\|_{L^2}^2 = \nu \left(1-\sqrt{\nu}\right),
\end{eqnarray*}
where we used that setting $\Lambda_{\bar{k}} = \max \{\Lambda_1, \ldots, \Lambda_{n-m}\} $ (note that $\Lambda_{\bar{k}} \neq0$) thanks to (\ref{johan6}) yields
\begin{eqnarray*}
\left( \langle I,J\rangle_{L^2} \right)^2 = \left(\int_0^{\tau} \sum_{k=1}^{n-m} \Lambda_k \tilde{s}_k^{x,p}(t) \, dt \right)^2& = &  \Lambda_{\bar{k}}^2  \left(\int_0^{\tau} \sum_{k=1}^{n-m} \frac{\Lambda_k}{\Lambda_{\bar{k}}} \, \tilde{s}_k^{x,p}(t) \, dt \right)^2 \\
& \leq &  \Lambda_{\bar{k}}^2 \, \nu \tau \, \int_0^{\tau} \left( \sum_{k=1}^{n-m} \frac{\Lambda_k}{\Lambda_{\bar{k}}} \, \tilde{s}_k^{x,p}(t)     \right)^2 \, dt \\
& = & \nu \, \|I\|_{L^2}^2  \, \|J\|_{L^2}^2.
\end{eqnarray*}
In conclusion, we deduce that  for every $(x,p)\in \mathcal{T}$, the function  $\omega=\omega^{x,p} \in L^2([0,1], \R^m)$ of the form (\ref{johan1}) with $\omega_1=\omega_1^{x,p}$ given by 
$$
\omega_1^{x,p}(t) := \frac{1_{[0,\tau]}  (t) P(t)}{ \|P\|^2_{L^2} } \qquad \forall t \in [0,1]
$$
satisfies (\ref{johan2}) and
$$
\left\| \omega_1^{x,p} \right\|_{L^2} \leq \frac{1}{\sqrt{ \nu \left(1-\sqrt{\nu}\right)}}.
$$
The proof of Lemma \ref{LEMaims} is complete.
\end{proof}

We conclude easily by compactness of $K$ and $\bar{U}$.
\end{proof}

\section{Proof of Theorem \ref{THMgen}}\label{SECproofTHMgen}

Let $M$ be an analytic compact manifold, $(\Delta,g)$ a two-step analytic sub-Riemannian structure and  $\mu$ a smooth measure on $M$. The following result, due to Agrachev and Lee \cite{al09} (see also \cite{riffordbook}), is a consequence of the fact that $\Delta$ is two-step (and the compactness of $M$). We refer the reader to \cite{al09,riffordbook} for the proof. In fact, it is worth mentioning that Agrachev and Lee prove that a sub-Riemannian structure is two-step if and only if $d_{SR}$ is locally Lipschitz in charts.

\begin{lemma}\label{LEMgen2}
The function $d_{SR}^2 : M\times M \rightarrow \R$ is locally Lipschitz in charts. In particular, there is $L>0$ such that $\left| \nabla_yf^x\right| \leq L$ for all $x\in M$ and $y \in \mathcal{O}_x$. Furthermore,  there is $A>0$ such that for every $x, y \in M$ and every minimizing geodesic $\gamma:[0,1] \rightarrow M$ from $x$ to $y$, there is $p\in T_x^*M$ with $|p|^*\leq A$ such that $\gamma$ is the projection of the normal extremal $\psi:[0,1] \rightarrow T^*M$ starting at $(x,p)$.
\end{lemma}

By the above lemma and Proposition \ref{PROPhsc}, there is $C>0$ such that for every $x\in M$ the function $f^x: y\mapsto d_{SR}(x,y)^2/2$ is $C$-nearly horizontally semiconcave in $M$. 
\begin{lemma}\label{LEMgen1}
There is $B>0$ such that for every $x\in M$ the following property holds: for every $y\in \mathcal{O}_x$, there is a neighborhood $U^y \subset \mathcal{O}_x$ of $y$ along with an orthonormal family of smooth vector fields $X^1,\ldots, X^m$ which parametrize $\Delta$ in $U^y$ such that
\begin{eqnarray}\label{22dec6}
\left\| X^i\right\|_{C^1} \leq B  \qquad \forall i=1, \ldots,m,
\end{eqnarray}
and
\begin{eqnarray}\label{22dec5}
\left[X^i \cdot (X^i \cdot f^x)\right](z)  \leq B \left|\nabla_z f^x\right| +B \qquad \forall z \in U^y, \, \forall i=1, \ldots,m,
\end{eqnarray}
where $\nabla_z f^x$ stands for the gradient of $f^x$ at $z$ with respect to the global Riemannian metric $g$.
\end{lemma}

\begin{proof}[Proof of Lemma \ref{LEMgen1}]
First of all, we notice that there is $A>0$ such that if $v^1, \ldots, v^m$ is an orthonormal family of tangent vectors in $\Delta(z)$ for some $z\in M$ then there is an orthonormal family of smooth vector fields $X^1, \ldots, X^m$ which generates the distribution $\Delta$ in a neighborhood of $z$ and such that $\|X^i\|_{C^1}$ is bounded by $A$ for all $i=1, \ldots, m$. Let $x\in M$ be fixed, then by $C$-nearly horizontal semiconcavity of $f^x$, for every $y\in M$, there are an open neighborhood $V^y$ of $0$ in $\R^m$, a function $\varphi^y: V^y \subset \R^m \rightarrow U$ of class $C^2$  and a function $\psi^y : V^y \subset \R^m \rightarrow \R$ of class $C^2$ such that (\ref{DEFnhscf1}) (with $f=f^x$), (\ref{DEFnhscf2}) and (\ref{DEFnhscf3}) are satisfied. Fix $y\in \mathcal{O}_x$ and define the function $F^y:U^y \rightarrow \R$ by $F^y:=f^x\circ \varphi^y - \psi^y$, it is of class $C^2$ and satisfies
$$
d_0F^y=0 \quad \mbox{and} \quad \mbox{Hess}_0 F^y \leq 0.
$$
Taking a chart near $y$ we can assume that we work in $\R^n$. Let $\varphi^y=(\varphi_1^y, \ldots, \varphi_n^y)$ and $(x_1, \ldots, x_n)$ and $(v_1, \ldots, v_m)$ the coordinates respectively in $\R^n$ and $\R^m$. Then, we have 
$$
\frac{ \partial F^y}{\partial v_i}(0) = \left( \sum_{k=1}^n \frac{\partial f^x}{\partial x_k}(y) \frac{\partial \varphi_k^y}{\partial v_i}(0) \right) - \frac{\partial \psi^y}{\partial v_i} (0)=0 \qquad \forall i=1, \ldots, m
$$
and for every $i=1, \ldots, m$,
\begin{multline*}
\frac{\partial^2 F^y}{\partial v_i^2} (0) =     \left( \sum_{k,l =1}^n \frac{\partial^2 f^x}{\partial x_l \partial x_k} (y) \frac{\partial \varphi_k^y}{\partial v_i}(0)  \frac{\partial \varphi_l^y}{\partial v_i}(0)  \right) \\
+       \left(  \sum_{k=1}^n \frac{\partial f^x}{\partial x_k} (y)\frac{\partial^2 \varphi_k^y}{\partial v_i^2} (0)\right)   - \frac{\partial^2 \psi^y}{\partial v_i^2}(0) \leq 0,
\end{multline*}
which yields
\begin{eqnarray}\label{22dec2}
\sum_{k,l =1}^n \frac{\partial^2 f^x}{\partial x_l \partial x_k} (y) \frac{\partial \varphi_k^y}{\partial v_i}(0)  \frac{\partial \varphi_l^y}{\partial v_i}(0)  & \leq & \frac{\partial^2 \psi^y}{\partial v_i^2}(0) -  \sum_{k=1}^n \frac{\partial f^x}{\partial x_k} (y)\frac{\partial^2 \varphi_k^y}{\partial v_i^2} (0) \nonumber\\
& \leq & C+ C \left| \nabla_y f^x\right|.
\end{eqnarray}
By (\ref{DEFnhscf2})  and the observation made at the very beginning of this proof, there is an orthonormal family of smooth vector fields $X^1, \ldots, X^m$ which generates the distribution $\Delta$ in a neighborhood of $z$ and such that 
\begin{eqnarray}\label{22dec4}
\left\| X^i\right\|_{C^1} \leq A \quad \mbox{and} \quad d_0\varphi^y(e_i)=\frac{\partial \varphi^y}{\partial v_i}(0)=X^i(y) \qquad \forall i=1, \ldots, m.
\end{eqnarray}
Setting $X^i=\sum_{k=1}^n a^i_k \partial_k$, we check easily that 
$$
X^i\cdot f^x = \sum_{k=1}^n a^i_k \, \frac{\partial f^x}{\partial x_k}
$$
and
$$
X^i \cdot \left(X^i \cdot f^x\right) =  \sum_{k=1}^n \left( \sum_{l=1}^n a^i_l \frac{\partial a^i_k}{\partial x_l} \right) \, \frac{\partial f^x}{\partial x_k} +  \sum_{k=1}^n a^i_k \left( \sum_{l=1}^n a^i_l  \frac{\partial^2 f^x}{ \partial x_l \partial x_k}\right). 
$$
The last expression at $y$ yields, thanks to (\ref{22dec2}) and (\ref{22dec4}) (which implies $a^i_k(y)=\frac{\partial \varphi^y_k}{\partial v_i}(0)$ for all $i=1, \ldots, m$ and $k=1, \ldots, m$)
$$
\left[X^i \cdot (X^i \cdot f^x)\right](y)  \leq A^2 \left| \nabla_y f^x\right| + C+ C \left| \nabla_y f^x\right| \qquad \forall i=1, \ldots,m.
$$
We conclude easily by smoothness of $f^x$ in $\mathcal{O}_x$ with $U^y$ sufficiently small and $B>0$ sufficiently large.
\end{proof}

In order to prove Theorem \ref{THMgen}, we need to bound from above the divergence of $f^x$ over $\mathcal{O}_x$ for all $x$ in $M$. The following holds:

\begin{lemma}\label{LEMgen3}
There is $N>0$ such that  the following property holds:
\begin{eqnarray}\label{divfxNgen}
\mbox{div}_y^{\, \mu} \left( \nabla^h f^x \right)  \leq N \qquad \forall y \in \mathcal{O}_x, \, \forall x \in M.
\end{eqnarray}
\end{lemma}

\begin{proof}[Proof of Lemma \ref{LEMgen3}]
Let $x\in M$ and $y\in \mathcal{O}_x$ be fixed, by Lemma \ref{LEMgen1} there is a neighborhood $U^y \subset \mathcal{O}_x$ of $y$ along with an orthonormal family of smooth vector fields $X^1,\ldots, X^m$ which parametrize $\Delta$ in $U^y$ such that (\ref{22dec5}) holds. The horizontal gradient of $f^x$ in $U^y$ is given by
$$
\nabla^h_y f^x = \sum_{i=1}^m \left(X^i \cdot f^x\right) (y) X^i(y).
$$
So, we have 
$$
\mbox{div}_y^{\mu} \left( \nabla^h f^x \right) = \sum_{i=1}^m \left(X^i \cdot f^x\right) (y) \, \mbox{div}_y^{\mu} \left(X^i\right) +  \sum_{i=1}^m \left[ X^i \cdot (X^i \cdot f)\right] (y).
$$
The second term above (in the right-hand side) is bounded thanks to (\ref{22dec5}) and Lemma \ref{LEMgen2} and the first term is bounded by (\ref{22dec6})  and Lemma \ref{LEMgen2} (the quantities $\left(X^i \cdot f^x\right) (y) $ are indeed bounded by the fact that $d_{SR}(x,\cdot)$ is solution to the horizontal eikonal equation, see \cite{fr10}). 
The proof of Lemma \ref{LEMgen3} is complete.
\end{proof}

To conclude the proof of Theorem \ref{THMgen}, we observe that as the functions $f^x=d_{SR}(x,\cdot)^2/2$ are Lipschitz on $M$ the minimizing Sard conjecture is satisfied (by Proposition \ref{PROPSardminEQ}) and we note that by analyticity the the sets $\mathcal{O}_x$ are geodesically star-shaped.  Then we can apply Proposition \ref{PROPversus} together with Lemma \ref{LEMgen3}.

\section{Proof of Theorem  \ref{THMCarnot}}\label{SECproofTHMCarnot}

By Proposition \ref{PROPCarnoteq}, it is sufficient to show that (\ref{divfxNsph}) holds. By assumption the function $f^0:y\rightarrow d_{SR}(0,y)^2/2$ is locally Lipschitz on $\G\setminus \{0\}$. Thus for every relatively compact open neighrborhood $U$ of $S_{SR}(0,1)$ with $\bar{U} \subset \G\setminus \{0\}$, there is $A>0$ such that  for every $y \in \bar{U}$ and every minimizing geodesic $\gamma:[0,1] \rightarrow M$ from $0$ to $y$, there is $p\in T_0^*M$ with $|p|^*\leq A$ such that $\gamma$ is the projection of the normal extremal $\psi:[0,1] \rightarrow T^*M$ starting at $(0,p)$. Moreover, a Carnot group whose first layer is equipped with a left-invariant metric is an analytic manifold equipped with an analytic sub-Riemannian structure. Consequently, by  Proposition  \ref{PROPhsc}, the function $f^0:y\rightarrow d_{SR}(0,y)^2/2$  is $C$-nearly horizontally semiconcave in $\bar{U}$ and we can repeat the arguments used in the proof of Theorem \ref{THMgen} for $y\in \mathcal{O}_0\cap S_{SR}(0,1)$.

\appendix

\section{Notations}\label{Notations}

We list below the notations used throughout this paper, we refer the reader to the monographs \cite{abb12,cs04,montgomery02,riffordbook} for further details:
\begin{itemize}
\item $M$ is a smooth manifold of dimension $n\geq 3$.
\item $\Delta$ is a smooth totally nonholonomic distribution of rank $m< n$.
\item $g$ is a smooth metric over $\Delta$. Sometimes, we see $g$ as the restriction of a global Riemannian metric $g$ on $M$. We use the notation $\langle \cdot, \cdot \rangle$ instead of $g_x(\cdot,\cdot)$ and we denote the norm associated with $g$ by $|\cdot|$ (instead of $|\cdot|_x=g_x(\cdot,\cdot)^{1/2}$). $B_{r}(x)$ stands for the open geodesic ball of radius $r>0$ centered at $x$.  The dual norm associated with the Riemannian metric $g$ on each $T_x^*M$ is denoted  by $|p|^*=|p|_x^*$ for every $(x,p)\in T^*M$.
\item We call horizontal path any $\gamma:[0,1] \rightarrow M$ in $W^{1,2}$ which is almost everywhere tangent to $\Delta$. We denote by $W^{1,2}_{\Delta}([0,1],M)$ the set of horizontal paths $\gamma:[0,1] \rightarrow M$ endowed with the $W^{1,2}$-topology. 
\item For every $\gamma\in W^{1,2}_{\Delta}([0,1],M)$, we define the length of $\gamma$ (w.r.t. $g$) by $\mbox{length}^g(\gamma)= \int_0^1 |\dot{\gamma}(t)|\, dt$ and its energy (w.r.t. $g$) by $\mbox{energy}^g(\gamma)= \int_0^1 |\dot{\gamma}(t)|^2\, dt$.
\item For any $x,y \in M$, we denote by $d_{SR}(x,y)$ (resp. $e_{SR}(x,y)$) the infimum of lengths (resp. energies) of horizontal paths joining $x$ to $y$. We note that $e_{SR}=d_{SR}^2$. We denote the open ball and the sphere centered at $x$ with radius $r>0$ respectively by $B_{SR}(x,r)$ and $S_{SR}(x,r)$. The geodesic distance  $d_{SR}$ is said to be complete if the metric space $(M,d_{SR})$ is complete. In this case, all closed balls $\bar{B}_{SR}(x,r)$ are compact (for any $x\in M$ and any $r>0$).
\item We call minimizing geodesic from $x$ to $y$ any $\gamma\in  W^{1,2}_{\Delta}([0,1],M)$ with $\gamma(0)=x, \gamma(1)=y$ which minimizes the energy $e_{SR}(x,y)$ (and so the distance $d_{SR}(x,y)$), that is such that $\mbox{energy}^g(\gamma)=e_{SR}(\gamma)$. We note that if $d_{SR}$ is complete, then there exist minimizing geodesics between any pair of points. 
\item For every $x\in M$, we denote by $W^{1,2}_{\Delta,x}([0,1],M)$ the set of paths in $W^{1,2}_{\Delta}([0,1],M)$ starting at $x$ (that is $\gamma(0)=x$) and we define the end-point map 
$$
E_{\Delta}^x \, : \, W^{1,2}_{\Delta,x}([0,1],M) \, \longrightarrow \, M
$$
by $E_{\Delta}^x(\gamma)=\gamma(1)$. The infinite dimensional space $W^{1,2}_{\Delta,x}([0,1],M)$ has a smooth manifold structure and the end-point map $E_{\Delta}^x$  is smooth.
\item An horizontal path $\gamma \in W^{1,2}_{\Delta,x}([0,1],M)$ is called singular if it is singular with respect to the end-point map $E_{\Delta}^x$, that is if the differential $d_{\gamma}E_{\Delta}^{x,1}$ is not surjective. It is convenient to rewrite the definition of singular curves in term of singular controls. If the distribution $\Delta$ is parametrized by a family $\mathcal{F}$ of $k$ smooth vector fields $X^1, \ldots, X^k$ in a open neighborhood of $\gamma([0,1])$ and if $u\in L^2([0,1],\R^k)$ satisfies 
$$
\dot{\gamma} (t) = \sum_{i=1}^k u_i(t) \, X^i(\gamma(t)) \, \mbox{ for a.e. } t \in [0,1],
$$
then $\gamma$ is singular if and only if the control $u$ is a singular point of the smooth mapping (well-defined in an open set $\mathcal{U}$)
$$
E_{\mathcal{F}}^{x,1} \, : \, \mathcal{U} \subset L^{2}([0,1],\R^k) \,  \longrightarrow  \, M
$$
defined by
$$
E_{\mathcal{F}}^{x,1}(v) :=  \gamma_v(1) \qquad \forall v \in L^2([0,1],\R^k),
$$
where $\gamma_v$ is the curve in $W^{1,2}_{\Delta,x}([0,1],M)$ solution to the Cauchy problem
$$
\dot{\gamma}_v (t) = \sum_{i=1}^k v_i(t) \, X^i \left(\gamma_v(t)\right) \, \mbox{ for a.e. } t \in [0,1], \quad \gamma_v(0)=x.
$$
\item An horizontal path $\gamma  \in W^{1,2}_{\Delta,x}([0,1],M)$ is  singular if and only if it is the projection of an abnormal extremal $\psi: [0,1] \rightarrow T^*M$ that never intersects the zero section of $T^*M$, such that
$$
\dot{\psi}(t) = \sum_{i=1}^k u_i(t) \vec{h}^i(\psi(t)) \qquad \mbox{for a.e. } t \in [0,1],
$$
where $\mathcal{F}$ is a family of $k$ smooth vector fields $X^1, \ldots, X^k$ which parametrizes $\Delta$ in a open neighborhood of $\gamma([0,1])$ and $\vec{h}^1, \ldots, \vec{h}^k$ are the Hamiltonian vector fields associated canonically with $h^i(x,p)=p\cdot Xi(x)$ in $T^*M$. The curve $\psi$ is called an abnormal lift of $\gamma$ and $\gamma$ is said to be abnormal.
\item The Hamiltonian $H:T^*M \rightarrow \R$ associated with $(\Delta,g)$ is defined by
$$
H(x,p) := \frac{1}{2} \max \left\{ \frac{p(v)^2}{g_x(v,v)} \, \vert \, v \in \Delta(x) \setminus \{0\} \right\} \qquad \forall (x,p)\in T^*M,
$$
which coincides with 
$$
\frac{1}{2} \sum_{i=1}^m \left( p\cdot X^i(x)\right)^2,
$$
if $\Delta$ is parametrized locally by an orthonormal family $X^1, \ldots, X^m$. The Hamiltonian vector field $\vec{H}$ associated with $(\Delta,g)$ is the Hamiltonian vector field given by $H$ with respect to the canonical symplectic form on $T^*M$. In local coordinates $(x,p)$ the trajectories $\psi=(x,p)$ of $\vec{H}$ are solution to 
$$
\dot{x} = \frac{\partial H}{\partial p}(x,p), \quad \dot{p}=- \frac{\partial H}{\partial x}(x,p),
$$
we call them normal extremals. Any projection of a normal extremal is an horizontal path that is said to be normal. 
\item An horizontal path $\gamma$ is called strictly abnormal if it is abnormal (singular) and not normal.
\item For every $x\in M$, the exponential mapping $\exp_x:T_x^*M \rightarrow M$ associated with $(\Delta,g)$ at $x$ is defined by $\exp_x:=\pi (\psi_{x,p}(1))$ where $\psi_{x,p}$ is the trajectory of $\vec{H}$ starting at $(x,p)$ and $\pi :T_x^*M \rightarrow M$ is the canonical projection. 
\item A Carnot group $(\G, \star )$ of step $s$ is a simply connected Lie group whose Lie algebra $\mathfrak{g}=T_0\G$ (we denote by $0$ the identity element of $\G$) admits a nilpotent stratification of step $s$, {\it i.e.} 
\begin{eqnarray}\label{carnot0}
\mathfrak{g}  = V_1 \oplus \cdots \oplus V_s,
\end{eqnarray}
with
\begin{eqnarray}\label{carnot1}
\bigl[ V_1, V_j\bigr] = V_{j+1} \quad \forall 1\leq j \leq s, \quad V_s \neq\{0\}, \quad V_{s+1}= \{0\}.
\end{eqnarray}
By simple-connectedness of $\G$ and nilpotency of $\mathfrak{g}$, $\exp_{\G}$ is a smooth diffeomorphism, which allows to identify $\G$ with its Lie algebra $\mathfrak{g} \simeq \R^n$. If the first layer $V_1$ of $\G$ is equipped with a left-invariant metric, then there is a set of coordinates $(x_1, \ldots, x_n)$, a one-parameter family of dilations $\{\delta_{\lambda}\}_{\lambda >0}$ of the form 
$$
\delta_{\lambda} \left(x_1,\ldots, x_n\right) = \left( \lambda^{d_1} x_1,  \lambda^{d_2} x_2, \ldots, \lambda^{d_n} x_n \right)  \qquad \forall x \in \R^n,
$$ 
and a orthonormal family of left-invariant vector fields generating $V_1$ satisfying
$$
X^i \left( \delta_{\lambda}(x)\right) = \lambda^{-1} \, \delta_{\lambda} \left( X^i(x)\right) \qquad \forall \lambda >0, \, x \in \R^n.
$$
\item A function $f:U\rightarrow \R$ on a open set $U\subset M$ is called locally semiconcave if for every $x\in U$ there are a open neighborhood $V\subset U$ of $x$ and  $C>0$ such that for any $y\in V$ there is a function $\psi:M\rightarrow \R$ with $\|\psi\|_{C^2}\leq C$ such that  $f\leq V$ on $M$ and $f(y)=\psi(y)$. For every $y\in U$, $d_y^+f$ denotes the set of super-differentials of $f$ at $y$, it is the set of $\alpha \in T_x^*M$ for which there is a function of class $C^1$, $\psi:M\rightarrow \R$ such that $\psi\geq f$ on $M$,   $\psi(y)=f(y)$ and $d_y\psi=\alpha$.
\item If $f:U\rightarrow M$ is smooth on the open set $U\subset M$, $\nabla^hf$ denotes its horizontal gradient with respect to $(\Delta,g)$. For every $y\in U$, $\nabla^h_yf$ is defined as the unique $v\in \Delta(y)$ such that $d_yf (w)=\langle v,w\rangle$ for all $w\in \Delta(y)$. If $\Delta(y)$ is generated by an orthonormal family $X^1(y), \ldots, X^m(y)$, then $\nabla^h_yf= \sum_{i=1}^m (X^i \cdot f)(y) \, X^i(y)$.
\end{itemize}
 
 \section{Proof of Lemma \ref{benin4j}}\label{PROOFbenin}
 
 Let $\mathcal{K} \subset \R^l$ be a compact set and $h:[0,1] \times \mathcal{K} \rightarrow \R$ an analytic mapping such that $h(0,\kappa)=0$ for all $\kappa\in \mathcal{K}$. Let $\bar{\kappa}\in\mathcal{K}$ be fixed, then  by  \cite[Lemma 4.12 p.126]{dvdd88}, there are an integer $d>0$ and $\rho>0$ together with analytic functions  $a_1, \ldots, a_d$ on $B(\bar{\kappa},\rho)$ and  $b_1, \ldots, b_d$ on  $ [0,\rho] \times B(\bar{\kappa},\rho) $ such that 
 $$
 h(t,\kappa) = \sum_{k=1}^{d} a_k(\kappa) \, b_k (t,\kappa) \, t^k \qquad \forall (t,\kappa) \in  [0,\rho] \times B(\bar{\kappa},\rho) 
 $$
 and 
 $$
 b_k (0,\bar{\kappa})=1 \qquad \forall k =1, \ldots, d. 
 $$
 Hence, by compactness of $\mathcal{K}$ we infer  that there are an integer $\bar{d}>0$ and  $\bar{\rho}>0$ such that for every $\bar{\kappa} \in \mathcal{K}$ there are analytic functions  $a_1^{\bar{\kappa}}, \ldots, a_{\bar{d}}^{\bar{\kappa}}$ on $B(\bar{\kappa},\bar{\rho})$ and  $b_1^{\bar{\kappa}}, \ldots, b_{\bar{d}}^{\bar{\kappa}}$ on  $[0,\bar{\rho}] \times B(\bar{\kappa},\bar{\rho})$ such that
 $$
 h(t,\kappa) = \sum_{k=1}^{\bar{d}} a_k^{\bar{\kappa}}(\kappa) \, b_k^{\bar{\kappa}} (t,\kappa) \, t^k \qquad \forall (t,\kappa) \in  [0,\bar{\rho}] \times B(\bar{\kappa},\bar{\rho}) 
 $$
 and 
 $$
 b_k (t,\kappa) \in [1/2,1] \qquad  \forall (t,\kappa) \in  [0,\bar{\rho}] \times B(\bar{\kappa},\bar{\rho}) , \, \forall k =1, \ldots, \bar{d}. 
 $$
 
 Let $\tau\in (0,\bar{\rho})$ be fixed. By Cauchy-Schwarz inequality, for every $\kappa \in \mathcal{K}$ we have 
 $$
\left( \int_0^{\tau} h(t,\kappa) \, dt \right)^2 \leq  \tau \, \int_0^{\tau} h(t,\kappa)^2 \, dt
$$
with equality only if $\lambda_1 h^2(\cdot,\kappa)=\lambda_2 h(\cdot,\kappa)$ for some nonzero $(\lambda_1,\lambda_2) \in \R^2$. Since $h(0,\kappa)=0$ the case of equality may only happen whenever $h(\cdot, \kappa)\equiv0$ where there holds $\left( \int_0^{\tau} h(t,\kappa) \, dt \right)^2 \leq \nu \tau \, \int_0^{\tau} h(t,\kappa)^2 \, dt$ for all $\nu \in (0,1)$. In conclusion, by compactness of $\mathcal{K}$ if there is no $\nu \in (0,1)$ such that 
 $$
\left( \int_0^{\tau} h(t,\kappa) \, dt \right)^2 \leq \nu \, \tau \, \int_0^{\tau} h(t,\kappa)^2 \, dt \qquad \forall \kappa \in \mathcal{K},
$$
then there is a sequence $\{\kappa_l\}_l$ in $\mathcal{K}$ converging to some $\kappa \in \mathcal{K}$  such that
$$
h \left(\cdot,\kappa_l\right) \not\equiv 0 \, \, \forall l \quad \mbox{and} \quad \lim_{l \rightarrow \infty}  \frac{\left( \int_0^{\tau}  h \left(t,\kappa_l\right)  \, dt \right)^2}{ \tau \int_0^{\tau}  h \left(t,\kappa_l\right) ^2 \, dt} =  1.
$$
 Then there are analytic functions $a_1^{\kappa}, \ldots, a_{\bar{d}}^{\kappa}$, $b_1^{\kappa}, \ldots, b_{\bar{d}}^{\kappa}$ and a sequence $\{\bar{k}_l\}_l$ in $\{1, \ldots, \bar{d} \}$ such that for all $l$,
 $$
 a_{\bar{k}_l}^{\kappa} \left( \kappa_l\right) \neq 0 \quad \mbox{and} \quad \left| a_{k}^{\kappa} \left(\kappa_l\right) \right| \leq \left| a_{\bar{k}_l}^{\kappa} \left( \kappa_l\right) \right| \quad \forall k \in \left\{1, \ldots, \bar{d}\right\},
 $$
  which allows to write
   $$ 
    \lim_{l \rightarrow \infty}  \frac{\left( \int_0^{\tau}    \sum_{k=1}^{\bar{d}} \frac{a_k^{\kappa} \left(\kappa_l\right)}{a_{\bar{k}_l}^{\kappa}\left(\kappa_l\right)} \, b_k^{\kappa} \left(t,\kappa_l\right) \, t^k    \, dt \right)^2}{ \tau \int_0^{\tau}   \left( \sum_{k=1}^{\bar{d}} \frac{a_k^{\kappa} \left(\kappa_l\right)}{a_{\bar{k}_l}^{\kappa}\left(\kappa_l\right)} \, b_k^{\kappa} \left(t,\kappa_l\right) \, t^k \right) ^2 \, dt} =  1.
 $$
 Therefore, since all the quantities $ a_k^{\kappa} \left(\kappa_l\right) /a_{\bar{k}_l}^{\kappa}\left(\kappa_l\right)$ belong to $[-1,1]$ for all $k$ and are equal to $1$ for $k=\bar{k}_l$, and since the functions $b_1(\cdot,\kappa), \ldots, b_{\bar{d}}(\cdot, \kappa)$ are analytic, there are 
  $c_1, \ldots c_{\bar{d}}$ in $[-1,1]$ and $\bar{k} \in \{1, \ldots, \bar{d}\}$ with $c_{\bar{k}}=1$ such that 
$$
   \left( \int_0^{\tau}   \sum_{k=1}^{\bar{d}} c_k \, b_k^{\kappa} \left(t,\kappa \right) \, t^k  \, dt \right)^2 =  \tau \int_0^{\tau}  \left(  \sum_{k=1}^{\bar{d}} c_k \, b_k^{\kappa} \left(t,\kappa \right) \, t^k \right)^2 \, dt.
$$
This means that there is a  nonzero pair $(\lambda_1,\lambda_2) \in \R^2$  such that 
 $$
\lambda_1 \,   \left(  \sum_{k=1}^{\bar{d}} c_k \, b_k^{\kappa} \left(t,\kappa \right) \, t^k \right)^2 = \lambda_2 \,   \left(  \sum_{k=1}^{\bar{d}} c_k \, b_k^{\kappa} \left(t,\kappa \right) \, t^k \right) \qquad \forall t \in [0,\tau].
 $$
which implies that 
$$
 \sum_{k=1}^{\bar{d}} c_k \, b_k^{\kappa} \left(t,\kappa \right) \, t^k = 0 \qquad \forall t\in [0,\tau].
$$
Let $k_0 \in \{1, \ldots, \bar{d}\}$ be such that $c_k=0$ for all $k\in \{1, \ldots, \bar{d}-1\}$ with $k<k_0$. Then we have 
 $$
c_{k_0} \,  b_{k_0}  \left(t,\kappa \right) \, t^{k_0} +  t^{k_0} \left( \sum_{k=k_0+1}^{\bar{d}} c_k \, b_k \left(t,\kappa \right) \, t^{k-k_0}\right) = 0 \qquad \forall t\in [0,\tau],
$$
which is impossible because $b_0(0,\kappa)\in [1/2,1]$. The proof is complete.

\end{document}